\RequirePackage{amsmath}
\documentclass[10pt, reqno]{amsart}
\usepackage{amssymb,url, color, mathrsfs}
\usepackage[colorlinks=true, bookmarks=true, pdfstartview=FitH, pagebackref=true, linktocpage=true, linkcolor = magenta, citecolor = blue]{hyperref}
\usepackage[short,nodayofweek]{datetime}
\usepackage{longtable}
\usepackage{indentfirst, tabularx}
\usepackage[table]{xcolor}
\usepackage{float, hhline, colortbl,cite}
\usepackage{graphicx}
\newtheorem{theorem}{Theorem}[section]
\newtheorem{lemma}[theorem]{Lemma}
\newtheorem{proposition}[theorem]{Proposition}

\def\R{\mathbb R}

\def\ge{\geq}
\def\De{\Delta} 

\def\ep{\epsilon}
\def\pa{\partial}
\def\lt{\left}
\def\rt{\right}

\def\sp{\R^n}

\def\i0i{\int_0^\infty}

\numberwithin{equation}{section}

\newcounter{hypo}

\makeatletter
\@addtoreset{equation}{section}
\makeatother

\title[ Fractional elliptic problems]{Liouville type theorems for  fractional elliptic problems}
\author[A. T. Duong]{Anh Tuan Duong}
\address{Anh Tuan Duong\\Department of Mathematics, Hanoi National University of Education\\136 Xuan Thuy Street, Cau Giay District, Ha Noi, Viet Nam}
\email{tuanda@hnue.edu.vn}
\author[V.H.Nguyen]{Van Hoang Nguyen}
\address{Van Hoang Nguyen\\Institute of Mathematics, Vietnam Academy of Science and Technology, Ha Noi, Viet Nam}
\email{vanhoang0610@yahoo.com}
\begin{document}
	\begin{abstract} In this paper, we   establish  Liouville type theorems for  stable  solutions on the whole space $\mathbb R^N$  to the fractional  elliptic equation
		$$(-\Delta)^su=f(u)$$
where the nonlinearity is nondecreasing and convex.  We also obtain a classification of stable solutions to  the fractional Lane-Emden system 
		$$\begin{cases}
		(-\Delta)^s u = v^p \mbox{ in }\R^N \\
		(-\Delta)^s v = u^q \mbox{ in }\R^N
		\end{cases}$$
 with  $p>1$  and $ q>1$. In our knowledge, this is  the first classification result for stable solutions of the fractional Lane-Emden system in literature.
	\end{abstract}
	\maketitle
	\section{Introduction}
Let $s$ be a positive real number satisfying $0<s<1$. 
The fractional Laplacian is defined on the space of rapidly decreasing functions by 
\begin{align*}
(-\De)^s u(x) = c_{N,s} \lim_{\ep \to 0} \int_{\R^N\setminus B_\ep(x)} \frac{u(x) - u(y)}{|x-y|^{N+2s}} dy, 
\end{align*}
where $c_{N,s}$ is the normalization constant and $B_\epsilon(x)$ is the ball centered at $x$ with radius $\ep$.
Notice further  that, see e.g., \cite{MRS16},  $(-\De)^s u(x)$ is well-defined at any $x\in\mathbb R^N$ when  $u \in C^{2\sigma} (\mathbb R^N)\cap \mathcal L_s(\R^N)$ for some $\sigma >s$  with
\[
\mathcal L_s(\R^N)= \left\{u :\R^N \to \R; \, \int_{\R^N} \frac{|u(y)|}{(1 + |y|)^{N+2s}} dy < \infty \right\}.
\]
We are, in this paper, interested  in   the nonexistence of nontrivial nonnegative stable solutions to  the fractional elliptic equation
\begin{equation}\label{eq:ellipticeq}
(-\Delta)^su=f(u)\mbox{ in } \R^N
\end{equation}
with general nonlinearity
and the fractional Lane-Emden system
\begin{equation}\label{eq:system}
\begin{cases}
(-\Delta)^s u = v^p \mbox{ in }\mathbb{R}^N\\
(-\Delta)^s v = u^q\mbox{ in }\mathbb{R}^N
\end{cases}
\end{equation}
with $p>1$ and $q>1$. In what follows, all solutions of \eqref{eq:ellipticeq} and \eqref{eq:system} are considered in the space  $C^{2\sigma} (\mathbb R^N)\cap \mathcal L_s(\R^N)$ for some $\sigma >s$.

\subsection{Fractional elliptic equation}
The first topic in this paper is concerned with  the classification of nontrivial nonnegative stable solutions to the problem \eqref{eq:ellipticeq}. Here, a solution to \eqref{eq:ellipticeq}  is called stable if 
\begin{equation}\label{e511197}
\int_{\mathbb R^N}f'(u)\phi^2dx\leq \frac{c_{N,s}}{2}\int_{\mathbb R^N}\int_{\mathbb R^N}\frac{(\phi(x)-\phi(y))^2}{|x-y|^{N+2s}}dydx \mbox{ for all }\phi\in C_c^\infty(\mathbb R^N).
\end{equation}
In the local case $s=1$, the nonexistence of nontrivial stable solutions of \eqref{eq:ellipticeq} has been much studied in the last two decades. For instant, in the  typical cases $f(u)=|u|^{p-1}u$ or $f(u)=e^u$,  the classification of stable solutions to \eqref{eq:ellipticeq}  was completely established in the pioneering articles \cite{Far07,Far07_1}, see also \cite{Cow16,CF12,Dup11}.  In \cite{Far07}, among other things,  it was shown that the equation $$-\Delta u=|u|^{p-1}u \mbox{ in } \R^N$$ 
admits nontrivial stable solutions $u\in C^2(\R^N)$ if and only if $N\geq 11$ and 
$$p\geq \frac{(N-2)^2-4N+8\sqrt{(N-1)}}{(N-10)(N-2)}.$$
It was also proved in \cite{Far07_1} that the equation $$-\Delta u=e^u\mbox{ in } \R^N$$
has no nontrivial  stable solution when $N<10$. This result is sharp.  In \cite{WY12}, the author obtained an optimal classification of stable weak solutions to the H\'{e}non type elliptic equations.

 In the case of  general nonlinearities, the nonexistence of nontrivial bounded radial stable solutions to \eqref{eq:ellipticeq} was obtained when $N\leq 10$ and $f\in C^1(\mathbb{R})$, see \cite{CC04,Ville07}. In non-radial case,  the classification of bounded stable solutions to \eqref{eq:ellipticeq} in low dimension was established in \cite{FSV08,Dan05,DF10}. In particular, Dupaigne and Farina proved  that,  under the assumption $f\in C^1(\mathbb{R})$, $ f\geq 0$ and $N\leq 4$, any bounded stable solution $u\in C^2(\mathbb{R})$  of  \eqref{eq:ellipticeq} with $s=1$ must be constant.  In higher dimensions,  the Liouville type theorem for bounded stable solutions of \eqref{eq:ellipticeq} with $s=1$  has been  established in \cite{DF10}, see also \cite{DF09}. Let us recall the assumptions on the nonlinearity used in \cite{DF10}. Consider  $f\in C^0(\mathbb R^+)\cap C^2(\mathbb R^+_*)$. For  $t>0$, define
$$q(t)=\begin{cases}
\dfrac{(f')^2}{ff''}(t)&\mbox{ if }ff''(t)\not=0\\
+\infty&\mbox{ if }ff''(t)=0
\end{cases}.
$$
Assume that  there exists the limit
\begin{equation}\label{e411191}
q_0=\lim_{t\to 0^+}q(t)\in \overline{\mathbb{R}}.
\end{equation}
As shown in \cite{DF10} that when $f$ is nondecreasing, convex and $f(0)=0$, then $q_0\geq 1$.

The following classification in arbitrary dimension is proved in \cite{DF10}.

\noindent{\bf Theorem A. }{\it 	Let $s=1$, $f\in C^0(\mathbb R^+)\cap C^2(\mathbb R^+_*)$ is nondecreasing, convex, $f>0$ in $\mathbb R^+_*$ and \eqref{e411191} holds. Assume that $u\in C^2(\mathbb R^N)$ is a bounded, nonnegative stable solution to \eqref{eq:ellipticeq}. Then $u\equiv 0$ if one of the following conditions is satisfied:
	\begin{enumerate}
		\item $N< 10$.
		\item $N=10$ and $p_0<+\infty$, where $p_0$ is the conjugate exponent of $q_0$, i.e., 
		\begin{equation}\label{e411192}
		\frac{1}{p_0}+\frac{1}{q_0}=1
		\end{equation}
		\item $N>10$ and $p_0<p_c(N,1)$, where $p_0$ is given in \eqref{e411192} and 
		$$p_c(N,1)=\frac{(N-2)^2-4N+8\sqrt{(N-1)}}{(N-10)(N-2)}.$$
\end{enumerate}}
It is worth to mentioning that, the classification of stable solutions to quasilinear elliptic equation has been also investigated recently, see e.g., \cite{Le16,CES09,DFSV09}.

Let us now consider \eqref{eq:ellipticeq} in the nonlocal case $0<s<1$. A natural question in studying the equations with fractional Laplacian is that whether one can obtain similar classifications to the case of Laplace operator.  The pioneering work in the classification of stable solutions to the fractional Lane-Emden equation, i.e., \eqref{eq:ellipticeq} with $f(u)=|u|^{p-1}u$, is due to D\'{a}vila, Dupaigne and Wei \cite{DDW17} where the authors exploited the monotonicity  formula and some nonlinear integral estimates. This technique has been used and generalized in \cite{FW16,FW17,RC19} to some fractional  elliptic equations with polynomial nonlinearities and weights.

In the case of general nonlinearity, the authors of \cite{DS10} obtained a fractional version of a result  in \cite{DF10 }in low dimensional case. More precisely, under the same assumptions of $f$, i.e., $f\in C^1(\mathbb{R})$, $ f\geq 0$, it was  shown that \eqref{eq:ellipticeq} has no nontrivial  bounded stable solution in dimension $N\leq 2$ if $0<s<\frac{1}{2}$ and in dimension $N\leq 3$ if $\frac{1}{2}\leq s<1$. However, in order to obtain a fractional version of Theorem A, the techniques in \cite{DS10,DDW17} seem not applicable.  In this paper, we develop a new technique  which allows one to use a  non-compactly supported function as test function. From this technique and delicate nonlinear integral estimates on half space $\R^{N+1}_+$, we obtain  a nonexistence result of nontrivial stable solutions of \eqref{eq:ellipticeq}   given in  the following theorem.
\begin{theorem}\label{t:elliptic}
Let $0<s<1$.	Assume that $f\in C^0(\mathbb R^+)\cap C^2(\mathbb R^+_*)$ is nondecreasing, convex, $f>0$ in $\mathbb R^+_*$ and \eqref{e411191} holds. Then the problem  \eqref{eq:ellipticeq} has no nontrivial bounded nonnegative  stable solution if one of the following conditions is satisfied:
	\begin{enumerate}
		\item $N<10s$.
		\item $N=10s$ and $p_0<+\infty$, where $p_0$ is given in \eqref{e411192}.
		\item $N>10s$ and $p_0<p_c(N,s)$, where $p_0$ is given in \eqref{e411192} and 
		$$p_c(N,s)=\frac{(N-2s)^2-4sN+8s\sqrt{s(N-s)}}{(N-10s)(N-2s)}.$$
	\end{enumerate}
\end{theorem}
Remark that, recently, the authors of the present paper have obtained a nonexistence of stable solutions of the fractional Gelfand equation, i.e., $f(u)=e^u$ under the assumption that $N<10s$, see \cite{TuanHoang}. Very recently, the complete classification of stable weak solutions to the fractional Gelfand equation has been proved in \cite{HY20}.
\subsection{Fractional Lane-Emden system}
The second topic in this paper is to study the nonexistence of positive stable solutions to the fractional Lane-Emden system \eqref{eq:system}.  Motivated by \cite{Mon05,Cow13,FG13}, 
a positive solution $(u,v)\in \big(C^{2\sigma} (\mathbb R^N)\cap \mathcal L_s(\R^N)\big)\times \big(C^{2\sigma} (\mathbb R^N)\cap \mathcal L_s(\R^N)\big)$ of \eqref{eq:system} is called stable  if there are two positive functions $\zeta_1$ and $\zeta_2$ satisfying 
\begin{equation}\label{eq:stable}
\begin{cases}
(-\Delta)^s \zeta_1 = pv^{p-1}\zeta_2 \mbox{ in }\mathbb{R}^N\\
(-\Delta)^s \zeta_2 = qu^{q-1}\zeta_1\mbox{ in }\mathbb{R}^N
\end{cases}.
\end{equation}
In this topic, let us begin with the local case. When $s=1$,  \eqref{eq:system} is known as the  Lane-Emden system which has received considerably attention in recent years, see the  pioneering  works \cite{Mit93,Mit96,SZ96,SZ98} and recent results \cite{Sou09, Cow13, MY19,Hu15,Hu18,HZ16,HHM16}. Concerning the class of classical positive solutions, the well-known conjecture states that the system \eqref{eq:system} with $s=1$ admits  solutions if and only if 
$$p>0, q>0\mbox{ and }\frac{1}{p+1}+\frac{1}{q+1}>1-\frac{2}{N}.$$
This conjecture has been solved only for the radial solutions in any dimension, see \cite{Mit93,Mit96,SZ96,SZ98} and for the nonradial solutions in low dimensions $N\leq 4$, see \cite{Mit96, SZ98, Sou09}. Some partial results in higher dimensions  were also obtained in \cite{Sou09}.

Concerning the class of stable positive solutions of \eqref{eq:system} with $s=1$, a nonexistence result was shown in the pioneering work of Cowan \cite{Cow13}. 

\noindent{\bf Theorem B.}
{\it  Let $s=1$,
	$$t_+=\sqrt{\frac{pq(q+1)}{p+1}}+\sqrt{\frac{pq(q+1)}{p+1}-\sqrt{\frac{pq(q+1)}{p+1}}}$$
	and 
	$$t_-=\sqrt{\frac{pq(q+1)}{p+1}}-\sqrt{\frac{pq(q+1)}{p+1}-\sqrt{\frac{pq(q+1)}{p+1}}}.$$
	\begin{enumerate}
		\item[i) ]  Suppose  that  $2<q\leq p$ and
		\begin{equation}
		N-2-\frac{4p+4}{pq-1}t_+<0.
		\label{ecow}
		\end{equation}
		Then there has  no positive stable solution of \eqref{eq:system}. In particular,  there has no positive stable solution of \eqref{eq:system}  for any $2\leq q\leq p$ if $N\leq 10$.
		\item[ii) ] Suppose that   $1<q\leq 2$,  $ t_-<\frac{q}{2}$ and  \eqref{ecow}. Then there has no positive stable solution of \eqref{eq:system}.
	\end{enumerate}
}

The main idea used in \cite{Cow13} is a combination of stability inequality, comparison principle and bootstrap argument. After that, this idea was  exploited by many authors in studying various elliptic systems \cite{Hu15,Hu18,HZ16,HHM16,DP17,Duong17}. In \cite{Hu15}, the author has obtained a classification  of positive stable solutions to the  weighted Lane-Emden system 
$$\begin{cases}
-\Delta u=(1+|x|^2)^\frac{\alpha}{2} v^p\\
-\Delta v=(1+|x|^2)^\frac{\alpha}{2} u^q
\end{cases} \mbox{ in }\R^N,$$
 where $\alpha>0$ and $\frac{4}{3}<q\leq p$ or $1<q\leq\min (p,\frac{4}{3})$ with additional assumption. In \cite{HHM16}, the authors have established a Liouville type theorem for the Lane-Emden system with general weights
 $$\begin{cases}
 -\Delta u=\rho(x) v^p\\
 -\Delta v=\rho(x) u^q
 \end{cases} \mbox{ in }\R^N,$$
 where $\rho$  is a radial function satisfying $\rho(x)\geq C (1+|x|^2)^\frac{\alpha}{2}$ at infinity.
 In \cite{HHM16}, a new inverse comparison principle is introduced for bounded positive stable solutions in order to deal with the case $1<p\leq \frac{4}{3}$.  In particular, the range of nonexistence result in \cite{HHM16} is larger than that in \cite{Cow13,Hu15}. We should also mention some nonexistence results of positive stable solutions to elliptic systems involving advection terms have been also studied in \cite{Duong17,Hu18}  by developing the approach of Cowan.
 
To the best of our knowledge, there has no works in literature classifying stable solutions to fractional elliptic systems. In fact, some serious difficulties arise when one wants to classify positive stable solutions to systems involving the fractional Laplacian. The first one is that, in order to obtain an {\it a priori} estimate,  the standard test-function method does not work well with the fractional Laplacian since $(-\Delta)^s \phi$ is not, in general, compactly supported for $\phi\in C^\infty_c(\R^N)$. In addition, the bootstrap  argument-a key step to get better exponent in the study of Lane-Emden system- becomes a challenging problem since one has no estimates on compact sets in the nonlocal case and one needs to transform nonlinear integral estimates on half space $\R^{N+1}_+$ to that on $\R^N$. Besides that, one also needs to establish a comparison principle for  the system \eqref{eq:system}. 

In this paper, we obtain the following theorem.
\begin{theorem}\label{t}
Let $0<s<1$.
\begin{enumerate}
	\item 	If  $p\geq q>\frac{4}{3}$ and 
	\begin{equation}\label{e1011191}
	N<2s+\frac{4s(p+1)}{pq-1}t_+,
	\end{equation}
	then the system \eqref{eq:system} has no stable positive solution.
	\item  If $1<q\leq \min (p,\frac{4}{3})$, $t_-<\frac{q}{2}$ and \eqref{e1011191} holds, then the system \eqref{eq:system} has no stable positive solution.
\end{enumerate}
\end{theorem}
Notice that when $p=q$, the condition \eqref{e1011191} is equivalent to $p<p_c(N,s)$ where $p_c(N,s)$ is given in Theorem \ref{t:elliptic}. 

Let us now sketch the outline of the proof  of Theorems \ref{t:elliptic} and \ref{t}. Concerning  Theorem \ref{t:elliptic}, we divide the proof into two cases according to the range of $q_0$. First, when $q_0>\frac{N}{2s}$ then the conjugate exponent $p_0<\frac{N}{N-2s}$. In this case, we make use of some comparison  and the strong maximum principle to get the desired result. When $q_0\leq \frac{N}{2s}$ the proof becomes more involved. The difficulty arises when we need to compare integrals on half space $\R^{N+1}_+$ to $\R^N$. Let $u$ be a bounded  nonnegative stable solution of \eqref{eq:ellipticeq}  and $U$ be the extension of  $u$ in the sense of \cite{CS07}. We first establish an integral  estimate
\begin{align*}
\begin{split}
&\int_{\mathbb{R}^N}f^{\frac{1}{q_1}+2\alpha}(u)\eta^2dx\leq C\int_{\mathbb R^{N+1}_+}f^{2\alpha}(U)(|\nabla \bar{\eta}|^2+|L_s \bar{\eta}^2|)t^{1-2s}dxdt,
\end{split}
\end{align*}
where $\bar{\eta}\in C^\infty_c(\R_+^{N+1}), \eta(x)=\bar{\eta}(x,0)$, $q_1<q_0$, $\alpha\in [1,1+1/\sqrt{q_0})$ and $L_s$ is  the second order differential operator given in \eqref{e1811191}. We next control the right hand side of the inequality above  and choose suitably test function $\bar{\eta}$ to get
 \begin{align*}
\begin{split}
\int_{\mathbb{R}^N}f^{\frac{1}{q_1}+2\alpha}(u)\rho_{N+2s}\left(x/R\right)dx\leq CR^{-2s}\int_{\mathbb R^N} f^{2\alpha}(u(y))\rho_{N+2s}(y/R)dy,
\end{split}
\end{align*}
where $R>0$ and $\rho_{N+2s}(x)=(1+|x|^2)^{-\frac{N+2s}{2}}$. This is, in fact, the most difficult step in the proof.  Finally, some computations and the H\"{o}lder inequality give the desired result.

The proof of Theorem \ref{t}  consists of the following main  steps:

$\bullet$ Establish a stability inequality and give an {\it a priori} estimate of solutions.

$\bullet$ Prove a comparison principle between $u$ and $v$.

$\bullet$ Use  bootstrap argument to get  better result.

One obtains the stability inequality  in Lemma \ref{lem:stabilityineq} by similar argument as in the local case. Nevertheless, the {\it a priori} estimate  of solutions is more delicate since we must prove nonlinear integral estimate on the whole space $\R^N$, see Proposition \ref{dgtiennghiem}. In Proposition \ref{comparisonprinciple}, we establish a comparion principle in nonlocal setting which is more or less new. As mentioned above, a serious  difficulty in dealing with fractional system is that the bootstrap argument in the local case does not work in the nonlocal case.  To overcome this difficulty, we make use of the bootstrap argument for the extensions $U,V$ of $u,v$ on $\mathbb{R}^{N+1}_+$ in the sense \cite{CS07} and develop the technique  in \cite{TuanHoang} which allows us to reduce the estimates on half space $\mathbb{R}^{N+1}_+$ to that on $\R^N$ .
 
 The rest of this paper is organized as follows. In Section \ref{s2}, we give the proof of Theorem \ref{t:elliptic}. In Section \ref{s3}, we prove the stability inequality and the comparison principle for the system  \eqref{eq:system}. The proof of Theorem \ref{t} is given in Section \ref{s4}.

\section{Proof of Theorem \ref{t:elliptic}}\label{s2}
This section is devoted to prove Theorem \ref{t:elliptic}. We first notice that if $u$ is a nontrivial nonnegative solution of \eqref{eq:ellipticeq} then $u$ is a positive solution of \eqref{eq:ellipticeq} by the strong maximum principle for fractional Laplacian. So, in order to prove Theorem \ref{t:elliptic}, it is enough to show that the equation \eqref{t:elliptic} does not possess any positive stable solution under the conditions in this theorem. In what follows, $C$ denotes a generic positive constant which may change from line to line or even in the same line.

  Suppose, in contrary,  that $u$ is a positive  stable solution of \eqref{eq:ellipticeq}. We shall point out a contradiction by considering two different cases of $q_0$.

\noindent{\bf Case 1. $q_0>\frac{N}{2s}$.}

We have  $p_0<\frac{N}{N-2s}$. By using \eqref{e411191}, there exist $p<\frac{N}{N-2s}$ and  $c_1>0$ such that, see \cite[Formula 2.5]{DF10},
\begin{equation}\label{e511193}
f(t)\geq c_1t^p\mbox{ for  } t \mbox{ near }0. 
\end{equation}
Let $\varphi$ be a nonnegative function satisfying
\begin{equation*}
\begin{cases}
(-\Delta)^s\varphi&=c_1\varphi^p\mbox{ in }B_1\\
\varphi&=0\mbox{ in }\mathbb R^N\setminus B_1
\end{cases},
\end{equation*}
where $B_R$ denotes the ball centered at the origin of radius $R$.
It follows from Proposition $1.4$ in \cite{RS14} that $\varphi \in C^\beta(\R^N)$ for some $\beta \in (0,1)$. So, $\varphi$ is bounded. For $R>0$, put $\varphi_R=R^{\frac{-2s}{p-1}}\varphi(x/R).$ Hence, $\varphi_R$ satisfies 
\begin{equation*}
\begin{cases}
(-\Delta)^s\varphi_R&=c_1\varphi_R^p\mbox{ in }B_R\\
\varphi_R&=0\mbox{ in }\mathbb R^N\setminus B_R
\end{cases}
\end{equation*}
and 
\begin{equation}\label{e511192}
R^{N-2s}\|\varphi_R\|_{L^\infty(B_R)}\leq R^{N-2s-\frac{2s}{p-1}}\|\varphi\|_{L^\infty(B_1)}\to 0\mbox{ as }R\to \infty.
\end{equation}
We next prove that 
\begin{equation}\label{e511191}
u(x)\geq c_2|x|^{-N+2s}\mbox{ for }x\in \mathbb{R}^N\setminus B_1.
\end{equation}
Indeed, let $U$ be the extension of $u$  in the sense of \cite{CS07}, i.e., 
for  $(x,t)\in\mathbb R^{N+1}_+$
\begin{equation}\label{e1404201}
U(x,t)=\int_{\mathbb R^N}P_s(x-z,t)u(z)dz,
\end{equation}
where $P_s(x,t)$ is the Poisson kernel 
$$P_s(x,t)=C(N,s)\frac{t^{2s}}{(|x|^2+t^2)^{\frac{N+2s}{2}}}$$
and $C(N,s)$ is the normalization constant.
Then, $U\in C^2(\mathbb R^{N+1}_+)\cap C(\overline{\mathbb R^{N+1}_+})$, $t^{1-2s}\partial_t U\in C(\overline{\mathbb R^{N+1}_+})$ and 
\begin{align}\label{e511196}
\begin{cases}
-{\rm div}(t^{1-2s}\nabla U)=0&\mbox{ in }\mathbb R^{N+1}_+\\
U=u&\mbox{ on }\partial\mathbb R^{N+1}_+\\
-\lim\limits_{t\to 0}t^{1-2s}\partial _t U=\kappa_s(-\Delta)^su&\mbox{ on }\partial\mathbb R^{N+1}_+
\end{cases},
\end{align}
where $\kappa_s=\frac{\Gamma(1-s)}{2^{2s-1}\Gamma(s)}$  and $\Gamma$ is the usual Gamma function.

Let $\Upsilon$ be the extension of  the function $|x|^{-N+2s}$ as determined in \cite[Lemma 4.1]{Fall18}. Then $\Upsilon$ is bounded on $\partial B_1^+$ where
$$B^+_R:=\{(x,t)\in\mathbb{R}_+^{N+1};|x|^2+t^2<R^2\},\quad R >0.$$
We choose the constant $c_2=\min_{x\in \partial B^+_1}\Upsilon$. Denote by 
\begin{equation}\label{e1811191}
L_s=t^{2s-1}{\rm div} (t^{1-2s}\nabla)=\Delta_x+\partial^2_{tt}+\frac{1-2s}{t}\partial_t.
\end{equation}
Using \cite[Formula (4.2)]{Fall18} with $\alpha=\frac{2s-N}{2}$ and \eqref{e511196}, we have for  $(x,t)\in \mathbb{R}^{N+1}_+\setminus  B^+_1$ $$L_s(U- c_2\Upsilon)=L_s U=0.$$  
Put $W(x,t):=U(x,t)-c_2\Upsilon(x,t)$. Let $R>1$, the strong maximum principle implies that, for $(x,t)\in B_R^+\setminus  B^+_1$, $$W(x,t)\geq \min_{\partial (B_R^+\setminus B_1^+)}W =\min\{\min_{\pa B_R^+} W, \min_{\pa B_1^+} W, \min_{\{(x,0); 1< |x| < R\}} W\}.$$
If there is $x_0\in\R^N$, $1<|x_0|<R$ such that $$W(x_0,0)=\min_{\partial (B_R^+\setminus B_1^+)}W=\min_{B_R^+\setminus B_1^+}W,$$
then we  use the L'Hospital rule and then \cite[Formula (4.2)]{Fall18} with $\alpha=\frac{2s-N}{2}$ to arrive at
\begin{align*}
0\leq 2s\lim_{t\to 0^+}\frac{W(x_0,t)-W(x_0,0)}{t^{2s}}&=\lim_{t\to 0^+}t^{1-2s}\partial_t(U-c_2\Upsilon)(x_0,t)\\ 
&=-(-\Delta)^su(x_0) = -f(u(x_0)) < 0
\end{align*}
which is impossible. Thus, $W\geq \min\{\min_{\partial B^+_1}W,\min_{\partial B^+_R}W\}$  on $B_R^+\setminus B_1^+$. Fix $(x,t)\in \mathbb{R}_+^{N+1}\setminus B^+_1$,  for any $R>|(x,t)|$ we have
$$W(x,t)\geq \min\{\min_{\partial B^+_1}W,\min_{\partial B^+_R}W\}.$$
Letting $R\to +\infty$ in this inequality, we obtain 
$$W(x,t)\geq \lim_{R\to +\infty}\min\{\min_{\partial B^+_1}W,\min_{\partial B^+_R}W\}=\min\{\min_{\partial B^+_1}W,\lim_{R\to +\infty}\min_{\partial B^+_R}W\}\geq 0.$$
Here we have used the fact that $\min_{\partial B^+_1}W\geq 0$, $U\geq 0$ and $\lim\limits_{R\to +\infty}\sup_{\partial B_R^+} \Upsilon=0$. 
Consequently $U(x,t)\geq c_2\Upsilon(x,t)$ for $(x,t)\in \mathbb R^{N+1}_+\setminus B^+_1 $ which gives \eqref{e511191}.
 
It results from \eqref{e511192} and \eqref{e511191} that there exists $R>0$ so that 
$$u(x)\geq\varphi_R(x).$$
Furthermore, from \eqref{e511193} we get
$$f(\varphi_R)\geq c_1\varphi_R^p.$$
Hence,
$$(-\Delta)^s(u-\varphi_R)\geq f(u)-f(\varphi_{R})\geq 0.$$
This combined with the strong maximum principle gives 
\begin{equation}\label{e511195}
u(x)>\varphi_R(x)\mbox{ for all } x\in\mathbb{R}^N.
\end{equation}
Given any unit directional vector $e$, define $\varphi_{R,t}(x)=\varphi_R(x+te)$. We claim that 
$$u\geq \varphi_{R,t}\mbox{ for all } t\geq 0.$$
Put 
$$T=\sup\{t\in [0,+\infty); u> \varphi_{R,l}\mbox{ for all }l\in [0,t)\}.$$
Since $u>\varphi_R=\varphi_{R,0}$, we deduce that $T>0$.
Assume that $T<\infty$ then $u\geq \varphi_{R,T}$  and there is $x_0>0$ such that $u(x_0)= \varphi_{R,T}(x_0)$. However, we still have
$$(-\Delta)^s(u-\varphi_{R,T})\geq f(u)-f(\varphi_{R,T})\geq 0.$$
Thus, the strong maximum principle implies that $u>\varphi_{R,T}$ which is a contradiction.

\noindent{\bf Case 2. $q_0\leq\frac{N}{2s}.$}

Let $U$ be the extension of $u$ in the sense of \eqref{e511196}.
Let $\phi:\mathbb{R}\to\mathbb{R}$ be a convex function and 
$$\psi(u)=\int_0^u(\phi')^2dt,\;K(u)=\int_0^u\psi(t)dt.$$
Take a radial test function  $\bar{\eta}\in C_c^\infty(\mathbb{R}^{N+1})$, ${\rm supp}\bar{\eta}\subset \{(x,t)\in\mathbb{R}^{N+1};\;|x|^2+|t|^2\leq 4\}$ and put $\eta(x)=\bar{\eta}(x,0)$.
Using the weak form of the first equation in \eqref{e511196} with the test function $\psi(U)\bar{\eta}^2$, we get 
\begin{align*}
&\kappa_s\int_{\mathbb{R}^N}f(u)\psi(u)\eta^2dx=\int_{\mathbb R^{N+1}_+}\nabla U\cdot\nabla\left(\psi(U)\bar{\eta}^2\right)t^{1-2s}dxdt\\
&=\int_{\mathbb R^{N+1}_+}\psi'(U)|\nabla U|^2\bar{\eta}^2t^{1-2s}dxdt+\int_{\mathbb R^{N+1}_+}\psi(U)\nabla U\cdot\nabla\bar{\eta}^2t^{1-2s}dxdt\\
&=\int_{\mathbb R^{N+1}_+}(\phi'(U))^2|\nabla U|^2\bar{\eta}^2t^{1-2s}dxdt-\int_{\mathbb R^{N+1}_+}K(U)L_s\bar{\eta}^2t^{1-2s}dxdt,
\end{align*}
where $L_s$ is given in \eqref{e1811191} and in the last equality, we have used an integration by parts and the fact that 
$$\lim_{t\to 0}t^{1-2s}\partial_t\bar{\eta}=0.$$
As a consequence, 
\begin{align}\label{e511198}
\begin{split}
\int_{\mathbb R^{N+1}_+}(\phi'(U))^2|\nabla U|^2\bar{\eta}^2t^{1-2s}dxdt&\leq \int_{\mathbb R^{N+1}_+}K(U)|L_s\bar{\eta}^2|t^{1-2s}dxdt\\
&+\kappa_s\int_{\mathbb{R}^N}f(u)\psi(u)\eta^2dx.
\end{split}
\end{align}
We now exploit the stability inequality \eqref{e511197} with the test function $\phi(u)\eta$:
\begin{align}\label{e511199}
\kappa_s\int_{\mathbb{R}^N}f'(u)\phi^2(u)\eta^2dx\leq \kappa_s\|\phi(u)\eta\|^2_{\overset{.}{H}^s(\mathbb R^N)}\leq \int_{\mathbb R^{N+1}_+}|\left(\nabla(\phi(U)\bar{\eta})\right)|^2t^{1-2s}dxdt.
\end{align}
In addition, 
\begin{align}\label{e5111910}
\begin{split}
 &\int_{\mathbb R^{N+1}_+}|\left(\nabla(\phi(U)\bar{\eta})\right)|^2t^{1-2s}dxdt=\int_{\mathbb R^{N+1}_+}(\phi'(U))^2|\nabla U|^2\bar{\eta}^2t^{1-2s}dxdt\\
 &+\int_{\mathbb R^{N+1}_+}(\phi(U))^2|\nabla \bar{\eta}|^2t^{1-2s}dxdt-\frac{1}{2}\int_{\mathbb R^{N+1}_+}(\phi(U))^2L_s \bar{\eta}^2t^{1-2s}dxdt.
\end{split}
\end{align}
Combining \eqref{e511198},\eqref{e511199} and \eqref{e5111910}, we obtain 
\begin{align}\label{e5111911}
\begin{split}
&\kappa_s\int_{\mathbb{R}^N}(f'(u)\phi^2(u)-f(u)\psi(u))\eta^2dx\leq \int_{\mathbb R^{N+1}_+}K(U)|L_s\bar{\eta}^2|t^{1-2s}dxdt\\
&+\int_{\mathbb R^{N+1}_+}(\phi(U))^2|\nabla \bar{\eta}|^2t^{1-2s}dxdt-\frac{1}{2}\int_{\mathbb R^{N+1}_+}(\phi(U))^2L_s \bar{\eta}^2t^{1-2s}dxdt.
\end{split}
\end{align}
It follows from the convexity of  $\phi$  that $\psi (u)\leq \phi'(u)\phi(u)$ and then $K(u)\leq \frac{1}{2}\phi^2(u)$. Thus, \eqref{e5111911} gives
\begin{align}\label{e5111912}
\begin{split}
&\kappa_s\int_{\mathbb{R}^N}(f'(u)\phi^2(u)-f(u)\psi(u))\eta^2dx\leq C\int_{\mathbb R^{N+1}_+}(\phi(U))^2(|\nabla \bar{\eta}|^2+|L_s \bar{\eta}^2|)t^{1-2s}dxdt.
\end{split}
\end{align}
For any $\alpha\in [1,1+\frac{1}{\sqrt{q_0}})$, choose $\phi=f^\alpha$. Then,
using \cite[Formulas (36), (38)]{DF10}, one has $$f'(u)\phi^2(u)-f(u)\psi(u)\geq Cf'(u)\phi^2(u)\geq Cf^{\frac{1}{q_1}+2\alpha}(u) \mbox{ for some }q_1<q_0\mbox{ fixed }.$$
This estimate and \eqref{e5111912}  follow that 
\begin{align}\label{e5111913}
\begin{split}
&\int_{\mathbb{R}^N}f^{\frac{1}{q_1}+2\alpha}(u)\eta^2dx\leq C\int_{\mathbb R^{N+1}_+}f^{2\alpha}(U)(|\nabla \bar{\eta}|^2+|L_s \bar{\eta}^2|)t^{1-2s}dxdt.
\end{split}
\end{align}
 Let $\tilde{\phi}\in C_c^\infty(\mathbb{R})$ such that $\tilde{\phi}(t)=1$ when $|t|\leq 1$ and $\tilde{\phi}(t)=0$ when $|t|\geq 2$. Let $R$ and $ \tilde{R}$ be positive numbers. Put $$\Phi_{\tilde{R}}(x,t)=\tilde{\phi}\left(\frac{|x|+t}{\tilde{R}}\right),\;\zeta(x,t)=\left(1+|x|^2+t^2\right)^{\frac{-N+2s}{4}}$$ and  $$\bar{\eta}_{\tilde{R }}(x,t)=\Phi_{\tilde{R}}(x,t)\zeta(x,t)$$ for $(x,t)\in\mathbb R^{N+1}_+.$ Then $${\eta}_{\tilde{R }}(x):=\bar{\eta}_{\tilde{R }}(x,0)=\tilde{\phi}\left({|x|}/{\tilde{R}}\right)\rho^{\frac{1}{2}}_{N+2s}(x),$$
where $\rho_{N+2s}(x) =(1+ |x|^2)^{-(N+2s)/2}$. 

Replacing $\bar{\eta}(x)$ and $\eta(x)$ by $\bar{\eta}_{\frac{\tilde{R }}{R}}\left(\frac{x}{R},\frac{t}{R}\right)$ and ${\eta}_{\frac{\tilde{R }}{R}}\left(\frac{x}{R}\right)$ in \eqref{e5111913}, one obtains
\begin{multline}\label{e5111915}
\int_{\mathbb{R}^N}f^{\frac{1}{q_1}+2\alpha}(u){\eta}^2_{\frac{\tilde{R }}{R}}\left(\frac{x}{R}\right)dx\\\leq CR^{-2}\int_{\mathbb R^{N+1}_+}f^{2\alpha}(U)\left(\left|\left(\nabla \bar{\eta}_{\frac{\tilde{R }}{R}}\right)\left(\frac{x}{R},\frac{t}{R}\right)\right|^2+\left|\left(L_s \bar{\eta}^2_{\frac{\tilde{R }}{R}}\right)\left(\frac{x}{R},\frac{t}{R}\right)\right|\right)t^{1-2s}dxdt.
\end{multline}
On the other hand, it follows from the Jensen inequality that 
\begin{equation*}
f^{2\alpha}(U(x,t))\leq \int_{\mathbb R^N}P_s(x-y,t)f^{2\alpha}(u(y))dy.
\end{equation*}
Here  $P_s$ is the Poisson kernel given above. Therefore,
\begin{align}\label{e5111914}
\begin{split}
&I:=R^{-2}\int_{\mathbb R^{N+1}_+}f^{2\alpha}(U)\left(\left|\left(\nabla \bar{\eta}_{\frac{\tilde{R }}{R}}\right)\left(\frac{x}{R},\frac{t}{R}\right)\right|^2+\left|\left(L_s \bar{\eta}^2_{\frac{\tilde{R }}{R}}\right)\left(\frac{x}{R},\frac{t}{R}\right)\right|\right)t^{1-2s}dxdt\\
&=R^{N-2s}\int_{\mathbb R^{N+1}_+}f^{2\alpha}(U(Rx,Rt))\left(\left|\nabla \bar{\eta}_{\frac{\tilde{R }}{R}}\right|^2+\left|L_s \bar{\eta}^2_{\frac{\tilde{R }}{R}}\right|\right)t^{1-2s}dxdt\\
&\leq R^{N-2s}\int_0^\infty\int_{\mathbb R^N}\left(\int_{\mathbb R^N}P_s(Rx-y,Rt)f^{2\alpha}(u(y))dy\right)(|\nabla \bar{\eta}_{\frac{\tilde{R }}{R}}|^2+|L_s \bar{\eta}^2_{\frac{\tilde{R }}{R}}|)t^{1-2s}dxdt\\
&=R^{-2s}\int_{\mathbb R^N} f^{2\alpha}(u(y))\left(\int_0^\infty \int_{\mathbb R^N}P_s(x-\frac{y}{R},t)\left(\left|\nabla \bar{\eta}_{\frac{\tilde{R }}{R}}\right|^2+\left|L_s \bar{\eta}^2_{\frac{\tilde{R }}{R}}\right|\right)t^{1-2s}dxdt\right) dy,
\end{split}
\end{align}
where in the last equality we have used the Fubini theorem and the homogeneity of $P_s$.
Denote by 
\begin{multline}\label{e611193}
h_{\tilde{R }}(y):=\int_0^\infty \int_{\mathbb R^N}P_s(x-y,t)(|\nabla \bar{\eta}_{\tilde{R }}|^2+|L_s \bar{\eta}_{\tilde{R }}^2|)t^{1-2s} dxdt\\=C(N,s)\int_0^\infty \int_{\mathbb R^N}\frac{t(|\nabla \bar{\eta}_{\tilde{R }}|^2+|L_s \bar{\eta}_{\tilde{R }}^2|)}{(|x-y|^2+t^2)^{\frac{N+2s}{2}}} dxdt.
\end{multline}
We estimate the first integral
\begin{align}\label{e611192}
\begin{split}
I_1:&=\int_0^\infty \int_{\mathbb R^N}\frac{t|\nabla \bar{\eta}_{\tilde{R }}|^2}{(|x-y|^2+t^2)^{\frac{N+2s}{2}}} dxdt\\
&\leq\frac{1}{2}\int_0^\infty \int_{\mathbb R^N}\frac{t|\nabla \Phi_{\tilde{R }}|^2\zeta^2}{(|x-y|^2+t^2)^{\frac{N+2s}{2}}} dxdt+\frac{1}{2}\int_0^\infty \int_{\mathbb R^N}\frac{t|\nabla \zeta|^2\Phi^2_{\tilde{R }}}{(|x-y|^2+t^2)^{\frac{N+2s}{2}}} dxdt\\
&=:\frac{1}{2}J_1(y)+\frac{1}{2}J_2(y)\leq C(1+\tilde{R }^{-2s})\rho_{N+2s}(y),
\end{split}
\end{align}
 where $C$ is independent of $\tilde{R }$. Here, in the last inequality, we have used the following estimates from \cite[Proof of {\bf Step 2} of Lemma $2.4$]{TuanHoang}
 $$J_1(y)\leq C\rho_{N+2s}(y)\mbox{ and }J_2(y)\leq C\tilde{R }^{-2s}\rho_{N+2s}(y).$$
 We now consider 
 $$I_2=\int_0^\infty \int_{\mathbb R^N}\frac{t|L_s \bar{\eta}_{\tilde{R }}^2|}{(|x-y|^2+t^2)^{\frac{N+2s}{2}}} dxdt.$$
 A straightforward computation gives 
 \begin{align*}
 L_s \bar{\eta}_{\tilde{R }}^2&=2|\nabla \bar{\eta}_{\tilde{R }}|^2+2\bar{\eta}_{\tilde{R }}L_s\bar{\eta}_{\tilde{R }}\\
&= 2|\nabla \bar{\eta}_{\tilde{R }}|^2+2\bar{\eta}_{\tilde{R }}(L_s \Phi_{\tilde{R }}\zeta+\Phi_{\tilde{R }}L_s\zeta+2\nabla\Phi_{\tilde{R }}\cdot\nabla\zeta)\\
&= 2|\nabla \bar{\eta}_{\tilde{R }}|^2+2\Phi_{\tilde{R }}L_s \Phi_{\tilde{R }}\zeta^2+2\Phi_{\tilde{R }}^2\zeta L_s\zeta+4\Phi_{\tilde{R }}\nabla\Phi_{\tilde{R }}\cdot\nabla\zeta\zeta
 \end{align*}
 Then, applying the Young inequality, we obtain
 \begin{align*}
 |L_s \bar{\eta}_{\tilde{R }}^2|\leq 6(\Phi^2_{\tilde{R }}|\nabla\zeta|^2+|\nabla\Phi_{\tilde{R }}|^2\zeta^2)+2\Phi_{\tilde{R }}|L_s \Phi_{\tilde{R }}|\zeta^2+\Phi_{\tilde{R }}^2\zeta |L_s\zeta|.
 \end{align*}
 It is now enough to decompose $I_2$ into four terms  and then observe that  the first and the last term are bounded by  $CJ_2$, the second and the third term are bounded by $CJ_1$. Hence,
 \begin{equation}\label{e611191}
 I_2\leq C(1+\tilde{R }^{-2s})\rho_{N+2s}(y),
 \end{equation}
 where $C$ is independent of $\tilde{R }$. Consequently,
 substituting \eqref{e611192} and \eqref{e611191} into \eqref{e611193} we arrive at
\begin{equation*}
h_{\tilde{R }}(y)\leq C(1+\tilde{R }^{-2s})\rho_{N+2s},
\end{equation*}
 where $C$ does not depend on $\tilde{R }$.
 This together with \eqref{e5111914} and \eqref{e5111915} implies that 
 \begin{align*}
 \begin{split}
 \int_{\mathbb{R}^N}f^{\frac{1}{q_1}+2\alpha}(u){\eta}^2_{\frac{\tilde{R }}{R}}\left(x/R\right)dx&\leq CR^{-2s}\int_{\mathbb R^N} f^{2\alpha}(u(y))h_{\frac{\tilde{R }}{R}}(y/R)dy\\
 &\leq CR^{-2s}(1+\tilde{R}^{-2s})\int_{\mathbb R^N} f^{2\alpha}(u(y))\rho_{N+2s}(y/R)dy.
 \end{split}
 \end{align*}
 Letting $\tilde{R }\to \infty$ in this estimate, we deduce that
  \begin{align}\label{e611194}
 \begin{split}
 \int_{\mathbb{R}^N}f^{\frac{1}{q_1}+2\alpha}(u)\rho_{N+2s}\left(x/R\right)dx\leq CR^{-2s}\int_{\mathbb R^N} f^{2\alpha}(u(y))\rho_{N+2s}(y/R)dy.
 \end{split}
 \end{align}
 Recall that that $q_1<q_0$ and $1\leq \alpha<1+\frac{1}{\sqrt{q_0}}$.  Applying the H\"{o}lder inequality in \eqref{e611194}, we get
\begin{multline*}
 \int_{\mathbb{R}^N}f^{\frac{1}{q_1}+2\alpha}(u)\rho_{N+2s}\left(x/R\right)dx\\\leq CR^{-2s}\left(\int_{\mathbb{R}^N}f^{\frac{1}{q_1}+2\alpha}(u)\rho_{N+2s}\left(x/R\right)dx\right)^{\frac{2\alpha}{1/q_1+2\alpha}}R^{N\frac{1/q_1}{1/q_1+2\alpha}},
\end{multline*}
or 
\begin{equation}\label{e611195}
\int_{\mathbb{R}^N}f^{\frac{1}{q_1}+2\alpha}(u)\rho_{N+2s}\left(x/R\right)dx\leq CR^{N-2s-4\alpha s q_1}.
\end{equation}
Under one of the conditions in Theorem \ref{t:elliptic}, we shall show that the exponent in the right hand side of \eqref{e611195} is negative by choosing $q_1$ close to $q_0$ and $\alpha$ close to $1+\frac{1}{q_0}$. Indeed,  it is enough to claim that 
\begin{equation}\label{e611196}
N-2s-4s(q_0+\sqrt{q_0})<0.
\end{equation}
We now consider $N<10s$. Then \eqref{e611196} is true since $q_0\geq 1$.
In the case $N=10s$ and $p_0<\infty$, \eqref{e611196} still holds since $q_0>1$. Finally, when $N>10s$, the condition $p_0<p_c(N,s)$ ensures that $q_0+\sqrt{q_0}-\frac{N-2s}{4s}>0$, i.e., \eqref{e611196} is also true.

Let $R\to +\infty$ in \eqref{e611195}, we obtain a contradiction.\qed

\section{Some technical lemmas  for fractional Lane-Emden system}\label{s3}
In this section, we prove the stability inequality for stable positive solutions and the comparison principle for positive solutions. The former is given in the following.
\begin{lemma}\label{lem:stabilityineq}
	Let  $(u,v)$ be a stable positive solution of \eqref{eq:system}. Then for all $\phi\in C_c^\infty(\mathbb R^N)$, we have 
	$$\sqrt{pq}\int_{\mathbb R^N}u^{\frac{q-1}{2}}v^{\frac{p-1}{2}}\phi^2dx\leq \frac{c_{N,s}}{2}\int_{\mathbb R^N}\int_{\mathbb R^N}\frac{(\phi(x)-\phi(y))^2}{|x-y|^{N+2s}}dydx.$$
\end{lemma}
\begin{proof}
The proof  is based on the idea of Cowan \cite{Cow13}. Let $\phi\in C_c^\infty(\mathbb R^N)$ be a test function. Multiplying the first equation in \eqref{eq:system} by $\frac{\phi^2}{\zeta_1}$, one has
\begin{equation*}
\int_{\mathbb R^N}pv^{p-1}\zeta_2\frac{\phi^2}{\zeta_1}dx=\int_{\mathbb R^N}(-\Delta)^s\zeta_1\frac{\phi^2}{\zeta_1}dx.
\end{equation*}
Using an integration by parts to the right hand side of this equality and a simple inequality $-a^2-b^2\leq -2ab$, one gets
\begin{align*}
&\int_{\mathbb R^N}pv^{p-1}\zeta_2\frac{\phi^2}{\zeta_1}dx=\frac{1}{2}\int_{\mathbb R^N}\left((-\Delta)^s\zeta_1\frac{\phi^2}{\zeta_1}+\zeta_1(-\Delta)^s\left(\frac{\phi^2}{\zeta_1}\right)\right)dx\\
&=\frac{c_{N,s}}{2}\int_{\mathbb R^N}\int_{\mathbb R^N}\frac{\phi^2(x)-\phi^2(x)\frac{\zeta_1(y)}{\zeta_1(x)}+\phi^2(x)-\phi^2(y)\frac{\zeta_1(x)}{\zeta_1(y)}}{|x-y|^{N+2s}}dydx\\
&\leq c_{N,s}\int_{\mathbb R^N}\int_{\mathbb R^N}\frac{\phi^2(x)-\phi(x)\phi(y)}{|x-y|^{N+2s}}dydx.
\end{align*}
 By exchanging the role of $x$ and $y$, we also have
 \begin{align*}
 \int_{\mathbb R^N}pv^{p-1}\zeta_2\frac{\phi^2}{\zeta_1}dx
 \leq c_{N,s}\int_{\mathbb R^N}\int_{\mathbb R^N}\frac{\phi^2(y)-\phi(y)\phi(x)}{|x-y|^{N+2s}}dydx.
 \end{align*}
As a consequence of the two inequalities above, there holds
 \begin{align}\label{eq:stabilityine1}
\int_{\mathbb R^N}pv^{p-1}\zeta_2\frac{\phi^2}{\zeta_1}dx
\leq \frac{c_{N,s}}{2}\int_{\mathbb R^N}\int_{\mathbb R^N}\frac{(\phi(x)-\phi(y))^2}{|x-y|^{N+2s}}dydx.
\end{align}
Similarly, we also deduce from the second equation of \eqref{eq:system} with the test function $\frac{\phi^2}{\zeta_2}$ that 
\begin{align}\label{eq:stabilityine2}
\int_{\mathbb R^N}qu^{q-1}\zeta_1\frac{\phi^2}{\zeta_2}dx
\leq \frac{c_{N,s}}{2}\int_{\mathbb R^N}\int_{\mathbb R^N}\frac{(\phi(x)-\phi(y))^2}{|x-y|^{N+2s}}dydx.
\end{align}
The inequalities \eqref{eq:stabilityine1} and \eqref{eq:stabilityine2} yield
\begin{align}\label{eq:stabilityine3}
\int_{\mathbb R^N}\left(pv^{p-1}\zeta_2\frac{\phi^2}{\zeta_1}+qu^{q-1}\zeta_1\frac{\phi^2}{\zeta_2}\right)dx
\leq  c_{N,s}\int_{\mathbb R^N}\int_{\mathbb R^N}\frac{(\phi(x)-\phi(y))^2}{|x-y|^{N+2s}}dydx.
\end{align}
Applying again the simple inequality $a^2+b^2\geq 2ab$ to the left hand side of \eqref{eq:stabilityine3}, we end the proof of Lemma.
	\end{proof}
From now on, we use the notation $\rho_m(x) = (1 + |x|^2)^{-\frac m2}$ for some $m>N$. We next establish an a {\it priori} estimate for positive solutions of \eqref{eq:system}.
\begin{proposition}\label{dgtiennghiem}
Suppose that $p, q >1$ and $(u,v)$ is a positive solution of \eqref{eq:system}. Let $R$ be a positive constant,  then it holds
\begin{align}\label{eq:priorest}
\begin{split}
&\int_{\R^N} v^p \rho_{N+2s}(x/R) dx \leq CR^{N- \frac{2s p(q+1)}{pq-1}}\\
& \int_{\R^N} u^q \rho_{N+2s}(x/R) dx \leq CR^{N- \frac{2s q(p+1)}{pq-1}}
\end{split}
\end{align}
for some constant $C$ depending only on $N,s,p$ and $q$.
\end{proposition}
Remark that on $B_R$ we have $\rho_{N+2s}(x/R) \sim C$ with some constant $C$ depending only on $N$ and $s$. Hence, this proposition provides not only the same estimates on $B_R$ which were obtained by Yang and Zou \cite{YZ19_b} but also the estimates outside the ball $B_R$.
\begin{proof}
Let $\varphi\in C^\infty_c(\R^N)$  be a cut-off function satisfying $0\leq \varphi\leq 1,$  $\varphi=1$ on $B_1$ and $\varphi=0$ outside $B_2$. For $k >0$, define $\varphi_k(x) = \varphi(x/k)$. Testing the system \eqref{eq:system} by $\varphi_k \rho_{N+2s} \in C_0^\infty(\R^N)$ we get
\[
\int_{\R^N} v^p \rho_{N+2s} \varphi_k dx = \int_{\R^N} u (-\Delta)^s (\varphi_k \rho_{N+2s}) dx 
\]
and 
\[
\int_{\R^N} u^q \rho_{N+2s} \varphi_k dx = \int_{\R^N} v (-\Delta)^s (\varphi_k \rho_{N+2s}) dx.
\]
Letting $k\to \infty$ and using \cite[Lemma 2.2]{TuanHoang} and the Lebesgue dominated convergence theorem, we obtain
\[
\int_{\R^N} v^p \rho_{N+2s} dx = \int_{\R^N} u (-\Delta)^s (\rho_{N+2s}) dx 
\]
and 
\[
\int_{\R^N} u^q \rho_{N+2s} dx = \int_{\R^N} v (-\Delta)^s (\rho_{N+2s}) dx.
\]
It follows from this and  \cite[Lemma 2.1]{TuanHoang} that 
\begin{equation}\label{eq:a_1}
\int_{\R^N} v^p \rho_{N+2s} dx \leq C \int_{\R^N} u \rho_{N+2s} dx 
\end{equation}
and 
\begin{equation}\label{eq:a_2}
\int_{\R^N} u^q \rho_{N+2s} dx \leq C\int_{\R^N} v  \rho_{N+2s} dx,
\end{equation}
with $C$ depending only on $N$ and $s$. Notice that $\rho_{N+2s} = \rho_s \rho_{N+s}$ and $\rho_{N+s} \in L^1(\R^N)$. Hence by  the H\"older inequality, we have
\begin{align}\label{eq:b_1}
\int_{\R^N} u \rho_{N+2s} dx &\leq \lt(\int_{\R^N} u^\theta \rho_s^q \rho_{N+s} dx\rt)^{\frac1q} \lt(\int_{\R^N} \rho_{N+s} dx\rt)^{\frac{q -1}q} \notag\\
&\leq C \lt(\int_{\R^N} u^q  \rho_{N+(q +1)s} dx\rt)^{\frac1q}\notag\\
&\leq C \lt(\int_{\R^N} u^q  \rho_{N+2s} dx\rt)^{\frac1q},
\end{align}
with $C$ depending only on $N,s$ and $q$. Similarly, we also get
\begin{equation}\label{eq:b_2}
\int_{\R^N} u \rho_{N+2s} dx \leq C \lt(\int_{\R^N} v^p  \rho_{N+2s} dx\rt)^{\frac1p}
\end{equation}
with $C$ depending only on $N,s$ and $p$. From \eqref{eq:a_1}, \eqref{eq:a_2}, \eqref{eq:b_1} and \eqref{eq:b_2} and $pq >1$, we obtain
\begin{equation}\label{eq:main1}
\int_{\R^N} u^q  \rho_{N+2s} dx \leq C\mbox{ and } \int_{\R^N} v^p  \rho_{N+2s} dx \leq C,
\end{equation}
for some constant $C$ depending only on $N,s,p$ and $q$. Remark that the estimate \eqref{eq:main1} holds for any positive solution $(u,v)$ of \eqref{eq:system}. 

We now use a scaling argument. For any $R>0$, then the functions
\[
u_R(x) = R^{\frac{2s(p+1)}{pq-1}} u(R x)\mbox{ and } v_R(x) = R^{\frac{2s(q+1)}{pq-1}} v(Rx),
\]
also form a positive  solution of \eqref{eq:system}. By \eqref{eq:main1}, we get
\[
\int_{\R^N} u_R^q  \rho_{N+2s} dx \leq C\mbox{ and }\int_{\R^N} v_R^p  \rho_{N+2s} dx \leq C.
\]
Making the change of variables, we obtain \eqref{eq:priorest}.
\end{proof}
The comparison principle is given  in the following.
\begin{proposition}\label{comparisonprinciple}
	Let $p\geq q>1$ and $(u,v)$ be a positive solution to \eqref{eq:system}. Then, there holds, point-wise in $\mathbb R^N$,
	\begin{equation*}
	\frac{v^{p+1}}{p+1}\leq \frac{u^{q+1}}{q+1}.
	\end{equation*}
\end{proposition}
\begin{proof}
	For  the simplicity of notations, put $\tilde\sigma=\frac{q+1}{p+1}\leq 1$ and $l=\tilde\sigma^{-\frac{1}{p+1}}>1$. Then, we need to prove that 
	\begin{equation}\label{e3010194}
	w:=v-lu^{\tilde\sigma}\leq 0.
	\end{equation}
	We first point out that \begin{equation}\label{e3010192}
	(-\Delta)^su^{\tilde\sigma}\geq \tilde\sigma u^{\tilde\sigma-1}(-\Delta)^su.
	\end{equation} Indeed,  we have 
	\begin{align}\label{e3010191}
	(-\Delta)^su^{\tilde\sigma}(x)=\int_{\mathbb R^N}\frac{u^{\tilde\sigma}(x)-u^{\tilde\sigma}(y)}{|x-y|^{N+2s}}dy.
	\end{align}
In addition,  $f(t)=t^{\tilde\sigma}, t>0$, is concave. Then, 
	$$f(u(y))\leq f(u(x))+f'(u(x))(u(y)-u(x)),$$
	or
	$$u^{\tilde\sigma}(y)-u^{\tilde\sigma}(x)\leq \tilde\sigma u^{\tilde\sigma-1}(x)(u(y)-u(x)).$$
	Substituting this into \eqref{e3010191}, we obtain \eqref{e3010192}.
	
	As a consequence of \eqref{e3010192}, there holds
	\begin{equation}\label{e3010193}
	(-\Delta)^sw=(-\Delta)^sv-(-\Delta)^su^{\tilde\sigma}\leq u^q-l\tilde\sigma u^{\tilde\sigma-1}v^p.
	\end{equation}
	Moreover $u^q-l\tilde\sigma u^{\tilde\sigma-1}v^p=l^pu^{\tilde\sigma-1}((lu^{\tilde\sigma})^p-v^p)\leq 0$ on the set $\{x;w(x)\geq 0\}$. It then results from \eqref{e3010193} that 
	\begin{equation*}
	(-\Delta)^sw\leq 0\mbox{ on the set }\{x;w(x)\geq 0\}.
	\end{equation*}
	 Let $W$ be  the extension of $w$ in the sense of \eqref{e1404201}-\eqref{e511196}. 
	 By following the argument in \cite[Lemma 3.1]{YZ19_b} which is inspired by the idea in  \cite{QS12}, we also get $W\leq 0$. In particular, the restriction $w$ of $W$ is nonpositive  which implies \eqref{e3010194}. This completes the proof of Proposition.
	\end{proof}
\section{Proof of Theorem \ref{t}}\label{s4}
In this section, we prove Theorem \ref{t} by way of contradiction. Assume that $(u,v)$ is a positive stable solution of \eqref{eq:system}. 
Denote by $U$ and $V$ the extensions of $u$ and $v$ in the sense of \eqref{e1404201}-\eqref{e511196}, i.e.,
	 $U,V\in C^2(\mathbb R^{N+1}_+)\cap C(\overline{\mathbb R^{N+1}_+})$, $t^{1-2s}\partial_t U, t^{1-2s}\partial_t V\in C(\overline{\mathbb R^{N+1}_+})$ and 
\begin{align}\label{e3010195}
\begin{cases}
-{\rm div}(t^{1-2s}\nabla U)=0&\mbox{ in }\mathbb R^{N+1}_+\\
U=u&\mbox{ on }\partial\mathbb R^{N+1}_+\\
-\lim\limits_{t\to 0}t^{1-2s}\partial _t U=\kappa_s(-\Delta)^su&\mbox{ on }\partial\mathbb R^{N+1}_+
\end{cases}.
\end{align}
and 
\begin{align}\label{e3010196}
\begin{cases}
-{\rm div}(t^{1-2s}\nabla V)=0&\mbox{ in }\mathbb R^{N+1}_+\\
V=v&\mbox{ on }\partial\mathbb R^{N+1}_+\\
-\lim\limits_{t\to 0}t^{1-2s}\partial _t V=\kappa_s(-\Delta)^sv&\mbox{ on }\partial\mathbb R^{N+1}_+
\end{cases}.
\end{align}
In order to use the bootstrap argument, we need the following lemma.
\begin{lemma}
For any $t_-<\gamma<t_+$ and $\Phi\in C_c^\infty(\mathbb R^{N+1}_+)$, there holds
\begin{align}\label{e30101910}
\int_{\mathbb R^{N+1}_+}|\nabla (U^\gamma\Phi)|^2t^{1-2s}dxdt\leq C\int_{\mathbb R^{N+1}_+}U^{2\gamma}|\nabla \Phi |^2 t^{1-2s}dxdt.
\end{align}
Here $t_+$ and $t_-$ are given in Theorem B.
\end{lemma}
\begin{proof}
	Let $\Phi\in C^\infty_c(\mathbb R^{N+1}_+)$ be a test function and $t_-<\gamma<t_+$. Define $\phi(x)=\Phi(x,0)\in C^\infty_c(\mathbb R^N)$. Multiplying the first equation in \eqref{e3010195} by $U^{2\gamma-1}\Phi^2$ and then  integrating by parts, we have
	\begin{align}\label{e3010197}
	\begin{split}
	&\kappa_s\int_{\mathbb R^N}v^pu^{2\gamma-1}\phi^2dx=\int_{\mathbb R^{N+1}_+}\nabla U\cdot\nabla(U^{2\gamma-1}\Phi^2)t^{1-2s}dxdt\\
&=(2\gamma-1)\int_{\mathbb R^{N+1}_+}|\nabla U|^2U^{2\gamma-2}\Phi^2t^{1-2s}dxdt+\frac{2}{\gamma}\int_{\mathbb R^{N+1}_+}\Phi\nabla(U^{\gamma})\cdot\nabla \Phi U^\gamma t^{1-2s}dxdt\\
&=\frac{2\gamma-1}{\gamma^2}\int_{\mathbb R^{N+1}_+}|\nabla U^\gamma|^2\Phi^2t^{1-2s}dxdt+\frac{2}{\gamma}\int_{\mathbb R^{N+1}_+}\Phi\nabla(U^{\gamma})\cdot\nabla \Phi U^\gamma t^{1-2s}dxdt.
	\end{split}
	\end{align}
A straightforward computation leads to
\begin{align*}
&\int_{\mathbb R^{N+1}_+}|\nabla U^\gamma|^2\Phi^2t^{1-2s}dxdt=\int_{\mathbb R^{N+1}_+}|\nabla (U^\gamma\Phi)|^2t^{1-2s}dxdt\\
&-2\int_{\mathbb R^{N+1}_+}\Phi\nabla(U^{\gamma})\cdot\nabla \Phi U^\gamma t^{1-2s}dxdt-\int_{\mathbb R^{N+1}_+}U^{2\gamma}|\nabla \Phi |^2 t^{1-2s}dxdt.
\end{align*}
Plugging this into \eqref{e3010197}, one has
	\begin{align}\label{e3010198}
\begin{split}
&\kappa_s\int_{\mathbb R^N}v^pu^{2\gamma-1}\phi^2dx
=\frac{2\gamma-1}{\gamma^2}\int_{\mathbb R^{N+1}_+}|\nabla (U^\gamma\Phi)|^2t^{1-2s}dxdt\\
&-\frac{2(2\gamma-1)}{\gamma^2}\int_{\mathbb R^{N+1}_+}\Phi\nabla(U^{\gamma})\cdot\nabla \Phi U^\gamma t^{1-2s}dxdt-\frac{1}{\gamma^2}\int_{\mathbb R^{N+1}_+}U^{2\gamma}|\nabla \Phi |^2 t^{1-2s}dxdt\\
&\geq \left(\frac{2\gamma-1}{\gamma^2}-\frac{2(2\gamma-1)\epsilon}{\gamma^2}\right)
\int_{\mathbb R^{N+1}_+}|\nabla (U^\gamma\Phi)|^2t^{1-2s}dxdt\\
&-\left(\frac{1}{\gamma^2}+\frac{2\gamma-1}{2\gamma^2\epsilon}\right)\int_{\mathbb R^{N+1}_+}U^{2\gamma}|\nabla \Phi |^2 t^{1-2s}dxdt,
\end{split}
\end{align}
where in the last estimate, we have used $ab\leq a^2\epsilon+\frac{b^2}{4\epsilon}$ for any $\epsilon>0.$

Choosing the test function $u^\gamma\phi$ in the stability inequality  and then using the comparison principle, one gets
\begin{align}\label{e3010199}
\begin{split}
\kappa_s\sqrt{\frac{pq(q+1}{p+1})}\int_{\mathbb R^N}v^pu^{2\gamma-1}\phi^2dx&\leq \kappa_s\sqrt{pq}\int_{\mathbb R^N}v^\frac{p-1}{2}u^{\frac{q-1}{2}}u^{2\gamma}\phi^2dx\\
&\leq \kappa_s\|u^\gamma\phi\|_{\overset{.}{H}^s(\mathbb{R}^N)}\\
&\leq \int_{\mathbb R^{N+1}_+}|\nabla (U^\gamma\Phi)|^2t^{1-2s}dxdt,
\end{split}
\end{align}
where in the last inequality, we have used the fact that $U^{\gamma}\Phi$ has the trace $u^\gamma\phi$ on $ \partial \R^{N+1}_+$.
It results from \eqref{e3010198} and \eqref{e3010199} that 
\begin{multline*}
\left(\frac{2\gamma-1}{\gamma^2}-\frac{2(2\gamma-1)\epsilon}{\gamma^2}-\sqrt{\frac{pq(q+1}{p+1})}\right)
\int_{\mathbb R^{N+1}_+}|\nabla (U^\gamma\Phi)|^2t^{1-2s}dxdt\\
\leq C\int_{\mathbb R^{N+1}_+}U^{2\gamma}|\nabla \Phi |^2 t^{1-2s}dxdt.
\end{multline*}
Since $t_-<\gamma<t_+$ and $\epsilon$ is chosen small enough, we have $$\frac{2\gamma-1}{\gamma^2}-\frac{2(2\gamma-1)\epsilon}{\gamma^2}-\sqrt{\frac{pq(q+1}{p+1}})>0.$$
Consequently, we get \eqref{e30101910}. 
\end{proof}
\begin{proof}[End of the proof of Theorem \ref{t}]
	Denote by 
	$$k_s=\frac{N+2-2s}{N-2s}.$$
	Under the assumption on the exponent $q$, we fix a real positive number $\tau$ satisfying 
$$
	2t_-\leq 2\tau<q
$$
and let $m$ be a non-negative integer satisfying 
	$$\tau k_s^{m-1}<t_+\leq \tau k_s^{m}.$$
	
	We construct an increasing geometric sequence $$t_-<t_1< t_2<...<t_m<t_+$$ as follows
	$$ 2t_1=2\tau k,2t_2=2\tau kk_s, ...,2t_m=2\tau kk_s^{m-1}, $$
	where $k\in [1,k_s]$ is chosen such that $t_m$ is arbitrarily close to $t_+$. 
	
Take $\tilde{\Phi}\in C_c^\infty(\mathbb R^{N+1}_+)$ satisfying $\tilde{\Phi}=1$ on $B^+_1$ and $\tilde{\Phi}=0$ outside $B^+_2$. Recall that 
$$B^+_R=\{(x,t)\in\mathbb R^{N+1}_+;\;|x|^2+t^2<R^2\}.$$
	From \eqref{e30101910} and the Sobolev inequality, one has
	\begin{align*}
\begin{split}
		\left(\int_{B^+_1} U^{2t_mk_s}t^{1-2s}dxdt\right)^{\frac{1}{t_mk_s}}&\leq \left(\int_{B^+_{2}}U^{2t_mk_s}\tilde{\Phi}^{2k_s}t^{1-2s}dxdt\right)^{\frac{1}{t_mk_s}}\\
		&\leq 	\left(\int_{B^+_{2}}|\nabla (U^{t_m}\tilde{\Phi})|^2t^{1-2s}dxdt\right)^{\frac{1}{t_m}}\\
	&\leq C\left(\int_{B^+_{2}}U^{2t_m}|\nabla \tilde{\Phi} |^2 t^{1-2s}dxdt\right)^{\frac{1}{t_m}}\\
	&\leq C\left(\int_{B^+_{2}}U^{2t_{m-1}k_s} t^{1-2s}dxdt\right)^{\frac{1}{t_{m-1}k_s}}.
\end{split}
	\end{align*}
	By an induction argument, we arrive at
		\begin{align}\label{e3110191}
	\begin{split}
	\left(\int_{B^+_1} U^{2t_mk_s}t^{1-2s}dxdt\right)^{\frac{1}{t_mk_s}}&\leq C\left(\int_{B^+_{2^{m-1}}}U^{2\tau k} t^{1-2s}dxdt\right)^{\frac{1}{\tau k}}\\
	&\leq C\left(\int_{B^+_{2^{m-1}}}U^{2\tau k_s} t^{1-2s}dxdt\right)^{\frac{1}{\tau k_s}}\\
	&\leq C\left(\int_{B^+_{2^{m}}}U^{2\tau }|\nabla \Phi|^2 t^{1-2s}dxdt\right)^{\frac{1}{\tau }},
	\end{split}
	\end{align}
	where $\Phi\in C_c^\infty(\mathbb R^{N+1}_+)$, $\Phi=1$ on $B^+_{2^{m-1}}$ and $\Phi=0$ outside $B^+_{2^m}$. 
	
	Our task is next to transform the estimate the right hand side of \eqref{e3110191}  on half space $\R^{N+1}_+$ to that on $\R^N$. Applying
 the Jensen inequality, we get
\begin{equation*}
U^{2\tau}(x,t)\leq \int_{\mathbb R^N}P_s(x-y,t)(u(y))^{2\tau}dy,
\end{equation*}
where $P_s$ is the Poisson kernel. 
Consequently,
\begin{align}\label{e1510198}
\begin{split}
&\int_{\mathbb R_+^{N+1}}U^{2\tau}|\nabla \Phi|^2t^{1-2s}dxdt\\
&\leq \int_0^\infty\int_{\mathbb R^N}\left(\int_{\mathbb R^N}P_s(x-y,t)(u(y))^{2\tau}dy\right)|\nabla \Phi|^2t^{1-2s}dxdt\\
&=\int_{\mathbb R^N} (u(y))^{2\tau}\left(\int_0^\infty \int_{\mathbb R^N}P_s(x-y,t)|\nabla \Phi|^2t^{1-2s} dxdt\right) dy.
\end{split}
\end{align}
Let us put 
\begin{multline*}
\varphi(y)=\int_0^\infty \int_{\mathbb R^N}P_s(x-y,t)|\nabla \Phi|^2t^{1-2s} dxdt\\=p(N,s)\int_0^\infty \int_{\mathbb R^N}\frac{t|\nabla \Phi|^2}{(|x-y|^2+t^2)^{\frac{N+2s}{2}}} dxdt.
\end{multline*}
Recall that $\Phi\in C^\infty_c(\mathbb R^{N+1}_+)$ and $\Phi=0$ outside $B^+_{2^{m}}$. In particular, $\mbox{supp}|\nabla\Phi|=0$ outside $B^+_{2^{m}}$. In order to estimate $\varphi$, we need only consider the integral in the set $B^+_{2^{m}}$.

Firstly, it is easy to see that $\varphi$ is continuous on $\mathbb R^N$ and $\varphi(y)>0$ for all $y\in\mathbb R^N$. Consequently, for   $|y|\leq 2^{m+1}$, we have $\varphi\sim C\rho_{N+2s}$. In addition, when $|y|>2^{m+1}$ and $|x|\leq 2^m$, there holds $$\frac{y}{2}\leq |x-y|\leq 2y.$$ Thus, we obtain for $|y|\geq 2^{m+1}$ that
\begin{align*}
C_1 \rho_{N+2s}(y)\int_0^\infty \int_{\mathbb R^N}|\nabla \Phi|^2dxdt\leq \varphi(y)\leq C_2 \rho_{N+2s}(y)\int_0^\infty \int_{\mathbb R^N}|\nabla \Phi|^2dxdt
\end{align*}
These above estimates  imply that,   for all $y\in \R^N$, 
\begin{align}\label{e3110192}
C_1 \rho_{N+2s}(y)\leq \varphi(y)\leq C_2 \rho_{N+2s}(y),
\end{align}
for some $C_1,C_2>0$.
We deduce from \eqref{e3110191}, \eqref{e1510198} and \eqref{e3110192} that 
\begin{align}\label{e3110193}
\left(\int_{B^+_1} U^{2t_mk_s}t^{1-2s}dxdt\right)^{\frac{1}{t_mk_s}}&\leq C\left(\int_{\mathbb R^N} (u(y))^{2\tau}\rho_{N+2s}(y) dy\right)^{\frac{1}{\tau}}.
\end{align}
We next exploit a scaling argument. Let $R$ be a large positive parameter. Then, as above, $(u_R,v_R)$  is also a stable positive solution of \eqref{eq:system} with $u_R(x) = R^{\frac{2s(p+1)}{pq-1}} u(R x)$ and $v_R(x) = R^{\frac{2s(q+1)}{pq-1}} v(Rx)$. Replacing $U$ and $u$ in \eqref{e3110193} by $U_R:=R^{\frac{2s(p+1)}{pq-1}} U(R x,Rt)$ and $u_R$, we arrive at
\begin{align*}
\left(\int_{B^+_1} U^{2t_mk_s}(Rx,Rt)t^{1-2s}dxdt\right)^{\frac{1}{t_mk_s}}&\leq C\left(\int_{\mathbb R^N} (u(Ry))^{2\tau}\rho_{N+2s}(y) dy\right)^{\frac{1}{\tau}}.
\end{align*}
This estimate implies that 
\begin{align*}
\left(R^{-N-2+2s}\int_{B^+_R} U^{2t_mk_s}(x,t)t^{1-2s}dxdt\right)^{\frac{1}{t_mk_s}}&\leq C\left(\int_{\mathbb R^N} (u(Ry))^{2\tau}\rho_{N+2s}(y) dy\right)^{\frac{1}{\tau}}\\
&\leq C\left(\int_{\mathbb R^N} (u(Ry))^{q}\rho_{N+2s}(y) dy\right)^{\frac{2}{q}}\\
&=C\left(R^{-N}\int_{\mathbb R^N} (u(y))^{q}\rho_{N+2s}(y/R) dy\right)^{\frac{2}{q}}\\
&\leq CR^{-\frac{4s(p+1)}{pq-1}}.
\end{align*}
Here we have used the H\"{o}lder inequality in the penultimate inequality and Proposition \ref{dgtiennghiem} in the last one. Hence, 
\begin{align}\label{e3110194}
\left(\int_{B^+_R} U^{2t_mk_s}(x,t)t^{1-2s}dxdt\right)^{\frac{1}{t_mk_s}}&\leq CR^{N+2-2s-\frac{4s(p+1)}{pq-1}t_mk_s}.
\end{align}
Since $k$ is chosen so that $t_m$ is sufficiently close to $t_+$, the exponent in the right hand side of \eqref{e3110194} is negative thanks to the assumptions of Theorem \ref{t}. Let $R$ tend to infinity, we have a contradiction.
	\end{proof}
\section*{Acknowledgments} The first author was supported by Vietnam Ministry of Education and Training under grant number B2019-SPH-02.	
	

\end{document}